\documentclass[12pt]{amsart}
\usepackage{graphicx}
\usepackage[margin=1.3in]{geometry}
\usepackage{amsmath}
\usepackage{bbm}
\usepackage{amsthm}
\usepackage{mathtools}
\usepackage{enumitem}
\usepackage{tikz-cd}
\newcommand{\NN}{\mathbbm{N}}
\newcommand{\RR}{\mathbbm{R}}

\newcommand{\LL}{\mathcal{L}}

\newcommand{\MM}{\mathcal{M}}

\newcommand{\piXI}{\pi_X^{-1}}

\newcommand{\skprod}{X\times Y}

\newcommand{\holder}{\mathrm{H\ddot{o}lder}}

\newcommand{\preSum}[1]{\sum_{(\Bar{x},\Bar{y})\in F^{- #1}(x,y)}}

\newtheorem{theorem}{Theorem}[section]
\newtheorem{lemma}[theorem]{Lemma}
\newtheorem{proposition}[theorem]{Proposition}
\newtheorem{corollary}[theorem]{Corollary}
\newtheorem{example}[theorem]{Example}

\theoremstyle{definition}

\theoremstyle{remark}
\newtheorem{remark}{Remark}

\DeclareMathOperator{\diam}{diam}

\title{Equilibrium States for Non-Uniformly Expanding Skew Products}
\author{Gregory Hemenway}
\address{Dept.\ of Mathematics, University of Houston, Houston, TX 77204}
\email{ghemenwa@cougarnet.uh.edu}

\begin{document}
\date{\today}
\thanks{The author was partially supported by NSF DMS-1554794 and DMS-2154378.}
\pagenumbering{arabic}
 
\begin{abstract}
We study equilibrium states for non-uniformly expanding skew products, and show how a family of fiberwise transfer operators can be
used to define the conditional measures along fibers of the product. We prove that the
pushforward of the equilibrium state onto the base of the product is itself an equilibrium
state for a H\"older potential defined via these fiberwise transfer operators.
\end{abstract}

\maketitle

\section{Introduction and Main Results}

Let $X$ and $Y$ be compact, connected Riemannian manifolds. Let $F$ be a skew product on $X\times Y$; i.e. there are continuous maps $f\colon X\to X$ and $\{g_x\colon Y\to Y|\ x\in X\}$ such that \[F(x,y)=(f(x),g_x(y)). \]

In their 1999 paper \cite{DG99}, Denker and Gordin showed that if a fibred system, a class of sytems including skew products, is uniformly expanding and topologically exact along fibers, then given a H\"older potential $\varphi\colon \skprod\to\RR$, there is a unique equilibrium state on $\skprod$ that has conditionals defined by a fiberwise Gibbs property and whose transverse measure on $X$ is a Gibbs measure for a certain H\"older potential on $X$. In this paper, we will extend this to systems that allow for non-uniform expansion along fibers (see Theorem \ref{thm:A} below). To state the main theorem, we recall some prior results and notations. Castro and Varandas \cite{CV13} proved that for a certain class of non-uniformly expanding systems and H\"older continuous potentials, eigendata for the Ruelle operator \[
    \LL_\varphi\psi(x,y)=\sum_{(\Bar{x},\Bar{y})\in F^{-1}(x,y)}e^{\varphi(\Bar{x},\Bar{y})}\psi(\Bar{x},\Bar{y}).
\] acting on the space of H\"older potentials can be used to construct a unique equilibrium state $\mu$ on $\skprod$. 

We will refer to $X$ as the base and $\{Y_x=\{x\}\times Y\}_{ x\in X}$ as the fibers of the product since \[\skprod=\bigcup_{x\in X}\{x\}\times Y.\] Note that each fiber $Y_x$ can be identified with $Y$. We will make the necessary distinctions as needed. In Section~\ref{sec:transop}, we shall describe fiberwise transfer operators $\LL_x\colon C(Y_x)\to C(Y_{fx})$ defined such that for any $\psi\in  C(Y_x)$, \[
    \LL_x\psi(fx,y)=\sum_{\Bar{y}\in g_x^{-1}y}e^{\varphi(x,\Bar{y})}\psi(x,\Bar{y}).
\]  We will show that the pushforward of $\mu$ onto $X$ is an equilibrium state for the following potential: \begin{equation}
    \Phi(x)=\lim_{n\to\infty} \log\dfrac{\langle\LL_{x}^{n+1}\mathbbm{1},\sigma\rangle}{\langle\LL_{fx}^{n}\mathbbm{1},\sigma\rangle}
    \label{eqn:Phi}
\end{equation} for any probability measure $\sigma$ on $Y$. Note that this is the same potential studied by Pollicott and Kempton \cite{PK11} and later by Piraino \cite{P19}.

Our main result is the following.

\renewcommand{\thetheorem}{\Alph{theorem}}
\setcounter{theorem}{0}
\begin{theorem}\label{thm:A}
Let $X$ and $Y$ be compact connected Riemannian manifolds and $(\skprod,F)$ be a Lipschitz skew product with a uniformly expanding base map and non-uniformly expanding fiber maps satisfying assumptions \ref{ass:A1} and \ref{ass:A2} in Section \ref{sec:NUE} below. Let $\varphi$ be a H\"older continuous potential on $\skprod$ satisfying condition \eqref{ass:pot} in Section \ref{sec:ESexun} and $\mu$ be its corresponding equilibrium state. Then the following are true. \begin{enumerate}
    \item The potential $\Phi$ in equation \eqref{eqn:Phi} exists independent of $\sigma$, is H\"older continuous, and satisfies $P(\varphi)=P(\Phi)$.
    \item $\hat{\mu}=\mu\circ\piXI$ is the unique equilibrium state for $\Phi$.
    \item  There is a unique family of measures $\{\nu_x\colon x\in X\}$ such that $\nu_x(Y_x)=1$ and $$\LL_x^*\nu_{fx}=e^{\Phi(x)}\nu_x.$$
    \item $x\mapsto\nu_x$ is weak$^*$-continuous.
    \item Let $\hat{h}$ and $\hat{\nu}$ be the eigendata of $\LL_\Phi$, i.e. $\LL_\Phi^*\hat{\nu}=e^{P(\Phi)}\hat{\nu}$, $\LL_\Phi\hat{h}=e^{P(\Phi)}\hat{h}$, and $\int\hat{h}d\hat{\nu}=1$. Then the measures $\mu_x=\frac{h(x,\cdot)}{\hat{h}(\cdot)}\nu_x$ are probability measures on $Y_x$ such that \[\mu=\int_X\mu_x\ d\hat{\mu}(x).\]
\end{enumerate} 
\label{sec:ES}
\end{theorem}

In Section \ref{sec:setting}, we recall some background on skew products, non-uniformly expanding maps, and the fiberwise transfer operators. Example \ref{ex:ManPom} shows that the doubling map in the base and Mannevile-Pomeau maps in the fibers whose parameters vary continuously in $x$ satisfies conditions \ref{ass:A1} and \ref{ass:A2}. Almost constant potentials satisfy \eqref{ass:pot} so Theorem A applies to an open set of potentials for a broad class of systems. In Section \ref{sec:trans}, we prove part (1) of Theorem \ref{thm:A} using the Hilbert metric and contractions on cones similar to Piraino \cite{P19}.  This will require new arguments since the fiberwise maps $\{g_x\}$ are only non-uniformly expanding. We also prove (3) (see Theorem \ref{thm:fibMeas}). In Section \ref{sec:fibers}, we complete the proof of Theorem \ref{thm:A} by establishing (2), (4) (see Lemma \ref{lem:measCont}), and (5).

\subsection*{Acknowledgements}
I would like to thank my advisor, Dr. Vaughn Climenhaga, for many insightful discussions during the writing of this paper.

\section{Non-uniformly Expanding Skew Products}
\label{sec:setting}

\renewcommand{\thetheorem}{\arabic{section}.\arabic{theorem}}
Let $X$ and $Y$ be compact, connected Riemannian manifolds and denote by $d$ the $L_1$ distance on $\skprod$. Denote by $\pi_X$ and $\pi_Y$ the natural projection maps from $\skprod$ onto $X$ and $Y$, respectively. 

\subsection{Dynamics of Skew Products}
To understand the dynamics of $F$ on $\skprod$, define for any $n\geq 0$ and $x\in X$, $$g_x^n:=\ g_{f^{n-1}x}\circ\dots\circ g_x\colon Y_x\to Y_{f^nx}.$$ Then for any $(x,y)\in \skprod$, the behavior of this system can be investigated through the sequence $$F^n(x,y)=(f^n(x),g_x^n(y)).$$ For each $n\geq 0$, define the $n^{th}$-Bowen metric as \[d_n((x,y),(x',y'))=\max_{0\leq i\leq n}\{d(F^i(x,y),F^i(x',y'))\}.\] Also, denote the $n^{th}$-Bowen ball centered at $(x,y)$ of radius $\delta>0$ by $$B_n((x,y),\delta)=\{(x',y')\colon d_n((x,y),(x',y'))<\delta\}.$$

\subsection{Uniform Expansion in the Base}

A map $f\colon X\to X$ is \emph{uniformly expanding} if there exists $C, \delta_f>0$ and $\gamma>1$ such that \[d(f^n(x,y),f^n(x',y'))\geq C\gamma^n d((x,y),(x',y'))\] whenever $d_n((x,y),(x',y'))\leq \delta_f$. One can assume without loss of generality that $C=1$ by passing to an adapted metric. This reduces locally expanding to \[d(f(x,y),f(x',y'))\geq\gamma d((x,y),(x',y'))\] whenever $d((x,y),(x',y'))\leq \delta_f$.

\subsection{Non-uniform Expansion Along Fibers}
\label{sec:NUE}

 We shall assume that $F$ is a local homeomorphism and that the map $f:X\to X$ is uniformly expanding.  The following paragraph describes our assumptions of non-uniform expansion along the fibers on $\skprod$. Condition \ref{ass:A1} says that $F$ is uniformly expanding outside of some region $\mathcal{A}$ and not too contracting in $\mathcal{A}$. Thus, if $\mathcal{A}$ is empty, then everything is reduced to the uniformly expanding case. Condition \ref{ass:A2} ensures that every point has at least one preimage in the expanding region.
 
 Assume there is a continuous function $(x,y)\mapsto L(x,y)$ such that for every $(x,y)\in \skprod$, there is a neighborhood $U_{x,y}$ of $(x,y)$ so that $F|_{U_{x,y}}$ is invertible and \[d(F^{-1}(u_x,u_y),F^{-1}(v_x,v_y))\leq L(x,y)d((u_x,u_y),(v_x,v_y))\] for all $(u_x,u_y),(v_x,v_y)\in F(U_{x,y})$. Since $F$ is a local homeomorphism from a compact connected manifold onto itself, $F$ is a covering map. Similarly for $f\colon X\to X$ so $\hat{d}:=|f^{-1}x|$ is constant in $x$. Thus, if $\overline{d}:=\deg(F)$, then $d:=|g^{-1}_x(y)|$ is constant for all $x\in X$ and $y\in Y_x$ and $\overline{d}=\hat{d}d$. Additionally, we shall assume that there exist constants $\gamma>1$ and $L\geq 1$, and an open region $\mathcal{A}\subset \skprod$ such that 
 \begin{enumerate}[label=\upshape{(A\arabic{*})}]
    \item \label{ass:A1}  $L(x,y)\leq L$ for every $x\in\mathcal{A}$ and $L(x,y)<\gamma^{-1}$ for all $x\not\in\mathcal{A}$, and $L$ is close enough to 1 so that equation (\ref{eqn:avg}) below is satisfied.
    \item \label{ass:A2} There exists a finite covering $\mathcal{U}$ of $\skprod$ by open sets for which $F$ is injective such that $\mathcal{A}$ can be covered by $q<d$ elements of $\mathcal{U}$. Moreover, we assume that the elements of $\mathcal{U}$ are small enough to separate curves on $\skprod$ in the sense that if $c$ is a distance-minimizing geodesic on $\skprod$, then each element of $\mathcal{U}$ can intersect at most one curve in $F^{-1}(c)$.
\end{enumerate} 
Note that \ref{ass:A2} is a strengthened version of H2 from Castro and Varandas \cite{CV13}.

\begin{example}\label{ex:ManPom}
The Manneville--Pomeau map $y\mapsto y+y^{p+1}\mod \mathbbm{Z}\ (p>0)$ on $\mathbbm{S}^1$ is a classic example of a system that displays non-uniform expansion. Define a map $F\colon X\times Y\to X\times Y$ by taking the base map $f$ to be the doubling map on $\mathbbm{S}^1$ and Manneville--Pomeau maps $g_x(y)=y+y^{p(x)+1}\mod \mathbbm{Z}$ in the fibers where $p(x)>0$ varies continuously in the base point. Each of these maps has two branches so $d=2$. Note that $g_x'(y)>g_x'(0)=1$ for all $y\not=0$. Let $\mathcal{A}$ be any small neighborhood around $\mathbbm{S}^1\times\{0\}\subset\mathbbm{T}^2$. Then on $\mathcal{A}^c$ the product map is uniformly expanding. So $q=1$. Then $F(x,y)=(f(x),g_x(y))$ satisfies conditions \ref{ass:A1} and \ref{ass:A2} and thus Theorem \ref{thm:A} holds for this example.
\end{example}

\begin{lemma}\label{lem:FibPreimage}
 If $F$ satisfies \ref{ass:A1} and \ref{ass:A2}, then for any $x,x'\in X$ and $y,y'\in Y$, we can pair off the preimages of $g_x^{-1}(y)=\{y_1,\ldots,y_{d}\}$ and $g_{x'}^{-1}(y')=\{y_1',\ldots,y_{d}'\}$ where for any $k=1,2,\ldots,q$, \[d((x,y_k),(x,y_k'))\leq L_xd((fx,y),(fx',y'))\] while for any $k=q+1,\ldots,d$, \[d((x,y_k),(x',y_k'))\leq \gamma_x^{-1}d((fx,y),(fx',y')).\]
 \end{lemma}

\begin{proof}
Let $(x,y),(x',y')\in\skprod$ and $c$ be a distance-minimizing geodesic between these points. Let $g_x^{-1}(y)=\{y_1,\ldots,y_d\}$. Since $F$ is a covering map, we can uniquely lift $c$ to curves $c_1,\dots,c_d$ such that each $c_k$ starts at $y_i$ and $F(c_k)=c$ for all $k$. Then letting $y_k'$ be the other endpoint of $c_k$, we get a collection of preimages $g_{x'}^{-1}(y')=\{y'_1,\ldots,y'_d\}$. Cover each $c_k$ by domains of injectivity as in \ref{ass:A2}. Then at most $q$ of these balls can intersect $\mathcal{A}$ and each one intersects at most one of the curves $c_k$. Thus there are at most $q$ curves $c_k$ that intersect $\mathcal{A}$. Without loss of generality, we can assume that these are the first $q$ preimages. Applying \ref{ass:A1} gives the desired result.
\end{proof}

\subsection{Existence and Uniqueness of Equilibrium States}
\label{sec:ESexun}

We say $\varphi\colon X\times Y\to\RR$ is \emph{$\alpha$-H\"older continuous} for $\alpha>0$ if \[|\varphi|_\alpha:=\sup_{(x,y)\not=(x',y')}\frac{|\varphi(x,y)-\varphi(x',y')|}{d((x,y),(x',y'))^\alpha}<\infty.\] We denote by $ C^\alpha= C^\alpha(\skprod)$ the Banach space of $\alpha$-$\holder$ continuous functions on $\skprod$. The $n^{th}$ Birkhoff sum is defined as $S_n\varphi(x,y)=\sum_{k=0}^n\varphi\circ F^k(x,y)$. %We say that $\varphi$ has the \emph{Bowen property} if $\varphi$ is continuous and there is a constant $K$ such that $\sup_{n\geq 1}|S_n\varphi(x,y)-S_n\varphi(x',y')|\leq K$. 

We denote by $\MM(\skprod)$ the space of Borel probability measures on $\skprod$ and $\MM(\skprod,F)$ those that are $F$-invariant. Given a continuous map $F\colon \skprod\to\skprod$ and a potential $\varphi\colon\skprod\to\RR$, the variational principle asserts that \begin{equation}
\label{eqn:pressure}
    P(\varphi)=\sup\Big\{h_\nu(F)+\int\varphi\ d\nu \colon \nu\in\mathcal{M}(\skprod,F) \Big\}
\end{equation} 
where $P(\varphi)$ denotes the topological pressure of $F$ with respect to $\varphi$ and $h_\mu(F)$ denotes the metric entropy of $F$. An \emph{equilibrium state} for $F$ with respect to $\varphi$ is an invariant measure that achieves the supremum in the right-hand side of equation \eqref{eqn:pressure}. For uniformly expanding maps, every equilibrium state $\mu$ satisfies the Gibbs property: for any $\varepsilon>0$, there exists a $C>0$ such that \[C^{-1}\leq\dfrac{\mu\big(B_n((x,y),\varepsilon)\big)}{e^{-nP(\varphi)+S_n\varphi(x,y)}}\leq C\] for any $(x,y)\in\skprod$ and $n\in \NN$.

For our purposes in this paper, we fix a $\holder$ potential $\varphi\in  C^\alpha$ satisfying \begin{equation}\tag{P}\label{ass:pot}
    \sup\varphi-\inf\varphi<\varepsilon_\varphi\ \text{and\ } |e^\varphi|_\alpha<\varepsilon_\varphi e^{\inf\varphi}
\end{equation} for some $\varepsilon_\varphi>0$ satisfying the equations \eqref{eqn:cond} and \eqref{eqn:bound} below (see Section \ref{sec:}). Potentials that are almost constant satisfy condition \eqref{ass:pot}. Thus, Theorem \ref{thm:A} holds for measures of maximal entropy. We assume that $L$ is close enough to $1$ and $0<\varepsilon_\varphi<\log d -\log q$ so that \begin{equation}\label{eqn:cond}
    e^{\varepsilon_\varphi}\cdot\Bigg(\dfrac{(d-q)\gamma^{-\alpha}+qL^\alpha}{d}\Bigg)<1.
\end{equation} Choose $\varepsilon>0$ such that $\dfrac{de^\varepsilon e^{\varepsilon_\varphi}}{q}<1$. Let $\iota=\iota(\varepsilon,d,q)\in(0,1)$ be given by Lemma \ref{lem:bwords} below. Assume that $L$ is close enough to $1$ so that there is a $c>0$ satisfying \begin{equation}\label{eqn:avg}
    0<\gamma^{-(1-\iota)}L^\iota<e^{-2c}<1.
\end{equation} Under these assumptions, it is known that there is a unique equilibrium state $\mu$ for $\varphi$ on $X\times Y$.

\begin{lemma}
 If $F$ is topologically exact and satisfies \ref{ass:A1}, \ref{ass:A2}, and $\varphi$ satisfies $\sup\varphi-\inf\varphi<\log\overline{d}-\log q$, then there exists an expanding conformal measure such that $\LL_\varphi^*\nu=\lambda\nu$ and $supp(\nu)=\overline{X\times Y}$, where the spectral radius of $\LL_\varphi$, $\lambda:=r(\LL_\varphi)\geq \overline{d} e^{\inf\varphi}$. Moreover, $\nu$ is a non-lacunary Gibbs measure and has a Jacobian with respect to $F$ given by $J_\nu F=\lambda e^{-\varphi}$.

\end{lemma}

\begin{proof}
    See Theorem 4.1 in Varandas and Viana \cite{VV08}.
\end{proof}
We will not use the non-lacunary property or $J_\nu F$. For more details, see \cite{VV08}. 

\begin{theorem}
\label{CV13}
Let $F\colon \skprod\to \skprod$ be a local homeomorphism with Lipschitz continuous inverse and $\varphi\colon \skprod\to \RR$ be a H\"older continuous potential satisfying \ref{ass:A1}, \ref{ass:A2}, and \eqref{ass:pot}. Then the Ruelle-Perron-Frobenius operator has a spectral gap property in the space of H\"older continuous observables, there exists a unique equilbrium state $\mu$ for $F$ with respect to $\varphi$ and the density $d\mu/d\nu$ is H\"older continuous.
\end{theorem}

\begin{proof}
    See Theorem A in Castro and Varandas \cite{CV13}.
\end{proof}

Denote by $\hat{\mu}=\mu\circ\piXI$ the pushforward of the equilibrium state $\mu$ onto the base $X$. Throughout this paper, we shall refer to this measure as the transverse measure for our skew product.

\subsection{Fiberwise Transfer Operators for Skew Products}
\label{sec:transop}

As common in the literature, we will utilize Ruelle operators to study the equilibrium state on $(\skprod,F)$. Define the transfer operator $\LL_\varphi$ acting on $ C(\skprod)$ by sending $\psi\in  C(\skprod)$ to \[
    \LL_\varphi\psi(x,y)=\preSum{1}e^{\varphi(\overline{x},\overline{y})}\psi(\overline{x},\overline{y}).
\] Note that under the skew product representation of $F$, we may write \[\sum_{(\overline{x},\overline{y})\in F^{-1}(x,y)}e^{\varphi(\overline{x},\overline{y})}\psi(\overline{x},\overline{y})=\sum_{\overline{x}\in f^{-1}x}\sum_{\overline{y}\in g_{\overline{x}}^{-1}y}e^{\varphi(\overline{x},\overline{y})}\psi(\overline{x},\overline{y}).\] This gives rise to a fiberwise transfer operator on the fibers of $\skprod$. We disintegrate $\varphi$ and get the family of fiberwise potentials $\{\varphi_x(\cdot)=\varphi(x,\cdot)\}_{x\in X}$. For every $x\in X$, let $\LL_x\colon C(Y_x)\to C(Y_{fx})$ be defined by \[
    \LL_x\psi_x(y)=\sum_{\overline{y}\in g_x^{-1}y}e^{\varphi_x(\overline{y})}\psi_x(\overline{y})
\]for any $\psi\in  C(\skprod)$. We shall iterate the transfer operator by letting \[\LL_x^n=\LL_{f^{n-1}x}\circ\cdots\circ \LL_x\colon C(Y_x)\to C(Y_{f^nx}).\] Along with each of these fiberwise operators, we define its dual $\LL_x^*$ by sending a probability measure $\eta\in \mathcal{M}(Y_{fx})$ to the measure $\LL_x^*\eta\in \mathcal{M}(Y_x)$ such that for any $\psi\in  C(\skprod)$, \[\int \psi\,d(\LL_x^*\eta)=\int \LL_x\psi\,d\eta.\]

\section{A Potential for the Transverse Measure}

\label{sec:trans}

\label{sec:}

Piraino \cite{P19} shows that for subshifts of finite type, $\hat{\mu}=\mu\circ\pi_X^{-1}$ is an equilibrium state for the following potential \[\Phi(x)=\lim_{n\to\infty}\log \dfrac{\langle\LL_x^{n+1}\mathbbm{1},\sigma\rangle}{\langle\LL_{fx}^{n}\mathbbm{1},\sigma\rangle}\] where $\sigma$ is any probability measure supported on $Y$. We will show in Theorem \ref{thm:potential} that this potential exists in our setting. Furthermore, in Theorem \ref{thm:PhiHolder} we show that $\Phi$ is $\holder$ continuous.

\subsection{Birkhoff Contraction Theorem}
It is not hard to check that the Ruelle operator preserves the Banach space of H\"older continuous potentials $C^\alpha= C^\alpha(\skprod),\ 0<\alpha<1$. A subset $\Lambda\subset  C^\alpha$ is called a \emph{cone} if $a\Lambda=\Lambda$ for all $a>0$. A cone $\Lambda$ is \emph{convex} if $\psi+\zeta\in\Lambda$ for all $\psi,\ \zeta\in\Lambda$. We say that $\Lambda$ is a closed cone if $\Lambda\cup\{0\}$ is \emph{closed} with respect to the H\"older norm. We assume our cones are closed, convex, and $\Lambda\cap(-\Lambda)=\emptyset$. For any probability measure $\eta$ and H\"older potential $\psi$, let $\langle\psi,\eta\rangle=\int\psi d\eta$. Given a closed cone $\Lambda\subset C^\alpha$, we can define the dual cone $\Lambda^*=\{\eta\in (C^\alpha)^*\colon \langle\psi,\eta\rangle\geq 0\ \text{for\ all\ } \psi\in\Lambda\}$. For more on cones, see Section 4 from \cite{N04} or the appendices of \cite{P19}.

Define a partial ordering $\preceq$ on $ C^\alpha$ by saying $\phi\preceq \psi$ if and only if $\psi-\phi\in \Lambda\cup\{0\}$ for any $\phi,\psi\in  C^\alpha$. 
Let \[A=A(\phi,\psi)=\sup\{t>0\colon t\phi\preceq\psi\}\ \text{and\ } B=B(\phi,\psi)=\inf\{t>0\colon \psi\preceq t\phi\}.\] The \emph{Hilbert projective metric} with respect to a closed cone $\Lambda$ is defined as \[\Theta(\phi,\psi)=\log\frac{B}{A}.\]

The following lemma is useful when calculating distances in the Hilbert metric. For a proof, see Section 4 in \cite{N04}.

\begin{lemma}
Let $\Lambda$ be a closed cone and $\Lambda^*$ its dual. For any $\phi,\psi\in \Lambda$, \[\Theta(\phi,\psi)=\log\Bigg(\sup\Bigg\{\dfrac{\langle\phi,\sigma\rangle\langle\psi,\eta\rangle}{\langle\psi,\sigma\rangle\langle\phi,\eta\rangle}\colon \sigma,\eta\in \Lambda^*\ \text{and\ } \langle\psi,\sigma\rangle\langle\phi,\eta\rangle\neq 0\Bigg\}\Bigg).\]

\label{sec:hilMet}
\end{lemma}

The main idea of the proof of Theorem~\ref{sec:ES} is to find a cone on which the fiberwise transfer operator is a contraction. To accomplish this, we will need the Birkhoff Contraction theorem:

\begin{theorem}[Birkhoff \cite{B57}]
Let $\Lambda_1, \Lambda_2$ be closed cones and $\LL\colon \Lambda_1\to \Lambda_2$ a linear map such that $\LL \Lambda_1\subset \Lambda_2$. Then for all $\phi,\psi\in \Lambda_1$ \[\Theta_{\Lambda_2}(\LL\phi,\LL\psi)\leq \tanh\Big(\frac{\diam_{\Lambda_2}(\LL \Lambda_1)}{4}\Big)\Theta_{\Lambda_1}(\phi,\psi)\] where $\diam_{\Lambda_2}(\LL \Lambda_1)=\sup\{\Theta_{\Lambda_2}(\LL\phi,\LL\psi)\colon \phi,\psi\in \Lambda_1\}$ and $\tanh\infty=1$.
\label{thm: Birkhoff}
\end{theorem}

\subsection{Existence of $\Phi$}
We will use cones of the form \[\Lambda_K=\Lambda^\alpha_K=\{\psi\in  C^\alpha(\skprod)\colon\psi>0\ \text{and\ } |\psi|_\alpha\leq K\inf\psi\}\cup\{0\}.\] It can be shown that $\Lambda_K$ is a closed cone in $ C^\alpha$. For these cones, we get an alternate way of calculating distances in the Hilbert metric.

\begin{lemma}
For any $\phi,\psi\in \Lambda_K$, \[A(\phi,\psi)=\inf_{z_1,z_2,z_3\in\skprod}\dfrac{Kd(z_1,z_2)^\alpha\psi(z_3)-(\psi(z_1)-\psi(z_2))}{Kd(z_1,z_2)^\alpha\phi(z_3)-(\phi(z_1)-\phi(z_2))}\] and \[B(\phi,\psi)=\sup_{z_1,z_2,z_3\in\skprod}\dfrac{Kd(z_1,z_2)^\alpha\psi(z_3)-(\psi(z_1)-\psi(z_2))}{Kd(z_1,z_2)^\alpha\phi(z_3)-(\phi(z_1)-\phi(z_2))}.\]

\label{lem:hilMet}
\end{lemma}

\begin{proof}
See Lemma 4.2 in Castro and Varandas \cite{CV13}.
\end{proof}

Denote by \[\Lambda_K^x=\{\psi\in  C^\alpha(\skprod)\colon\psi_x(\cdot)>0\ \text{and\ } |\psi|_\alpha\leq K\inf\psi\}\cup\{0\}\] the cross section of $\Lambda_K$ that lives on $Y_x$.

Let \[s:=e^{\varepsilon_\varphi}\cdot\Bigg(\dfrac{(d-q)\gamma^{-\alpha}+qL^\alpha}{d}\Bigg)<1\] as in equation \eqref{eqn:cond}. We assume that $\varepsilon_\varphi>0$ is small enough that \begin{equation}\label{eqn:bound}
    \zeta:=s+2s\varepsilon_\varphi\diam(Y)^\alpha <1.
\end{equation} Then we have the following lemma based on arguments similar to Theorem 4.1 and Proposition 4.3 in \cite{CV13}. 

\begin{lemma}
With $\zeta$ as in \eqref{eqn:bound}, for all $K$ sufficiently large, we have $\LL_x(\Lambda_{K}^x)\subset\Lambda_{\zeta K}^{fx}$ for all $x\in X$. Moreover, there is a constant $M=M(K)>0$ such that for all $x\in X$, $\diam(\LL_x\Lambda^x_K)\leq M<\infty$ with respect to the Hilbert projective metric on $\Lambda_K^{fx}$.
\label{lem:ConeContraction}
\end{lemma}

\begin{proof}
Fix $x\in X$ and $K>0$. Denote by $\{y_k\}$ and $\{y_k'\}$ the preimages of $y$ and $y'$ in $Y_x$, respectively, as given by Lemma \ref{lem:FibPreimage}. Now fix $\psi\in\Lambda_K$. Since $\inf\LL_x\psi\geq de^{\inf\varphi}\inf\psi$ and $$\LL_{x}\psi(fx,y)-\LL_{x}\psi(fx,y') = \sum_{k=1}^d\Big(e^{{\varphi}(x,y_k)}(\psi(x,y_k)-\psi(x,y_k'))+(e^{{\varphi}(x,{y_k})}-e^{{\varphi}(x,y_k')})\psi(x,y_k')\Big),$$ we have \begin{align*}
    \dfrac{|\LL_x\psi(fx,y)-\LL_x\psi(fx,y')|}{\inf\LL_x\psi}&\leq d^{-1}\sum_{k=1}^d e^{\varphi(x,y_k)-\inf\varphi}\big|\psi(x,y_k)-\psi(x,y_k')\big|(\inf\psi)^{-1}\\&\quad+d^{-1}\sum_{k=1}^d(\sup\psi/\inf\psi)e^{-\inf\varphi}\big|e^{\varphi(x,y_k)}-e^{\varphi(x,y_k')}\big|=:I_1+I_2
\end{align*}

Note that $\big|\psi(x,y_k)-\psi(x,y_k')\big|\leq|\psi|_\alpha d(y_k,y'_k)^\alpha\leq K\inf\psi d(y_k,y'_k)^\alpha$. By Lemma \ref{lem:FibPreimage}, $d(y_k,y_k')\leq Ld(y,y')$ for any $1\leq k\leq q $ and $d(y_k,y_k')\leq \gamma ^{-1}d(y,y')$ for $q<k\leq d$ so \begin{align*}
    I_1&\leq d^{-1}\sum_{k=1}^d e^{\varphi(x,y_k)-\inf\varphi}Kd(y_k,y'_k)^\alpha\\
    &\leq d^{-1}e^{\varepsilon_\varphi}K\sum_{k=1}^d d(y_k,y'_k)^\alpha\\
    &\leq K e^{\varepsilon_\varphi}d^{-1}(L^\alpha q+(d-q)\gamma^{-\alpha})d(y,y')^\alpha\\
    &\leq sKd(y,y')^\alpha
\end{align*} where the second inequality holds by \eqref{ass:pot}.

To estimate $I_2$, note that $\big|e^{\varphi(x,y_k)}-e^{\varphi(x,y_k')}\big|\leq |e^{\varphi_x}|_\alpha d(x,y_k),(x,y_k'))^\alpha$ and \[\sup\psi\leq\inf\psi+|\psi|_\alpha\diam(Y)^\alpha\leq(1+K\diam(Y)^\alpha)\inf\psi\] implies that \[\sup\psi/\inf\psi\leq 1+K\diam(Y)^\alpha\leq 2K\diam(Y)^\alpha\] provided that $K$ is sufficiently large. Then \eqref{ass:pot} implies that \begin{align*}
    I_2&\leq 2K\diam(Y)^\alpha e^{-\inf\varphi}d^{-1}\sum_{k=1}^d|e^{\varphi_x}|_\alpha d(x,y_k),(x,y_k'))^\alpha\\
    &\leq 2K\diam(Y)^\alpha\varepsilon_\varphi d^{-1}\sum_{k=1}^d (L^\alpha q+(d-q)\gamma^{-\alpha})d(y,y')^\alpha\\
    &\leq 2K\diam(Y)^\alpha s\varepsilon_\varphi d(y,y')^\alpha\\
    &\leq 2s\varepsilon_\varphi\diam(Y)^\alpha K d(y,y')^\alpha
\end{align*} Therefore, if we let $\zeta:=s+2s\varepsilon_\varphi\diam(Y)^\alpha$, we have that \[|\LL_x\psi|_\alpha\leq (s+2s\varepsilon_\varphi\diam(Y)^\alpha)K\inf\LL_x\psi\leq \zeta K\inf\LL_x\psi\] so $\LL_x\psi\in\Lambda^{fx}_{\zeta K}$. 

Note that $\sup\LL_x\psi\leq(1+\zeta K(\diam Y)^\alpha)\inf\LL_x\psi$. Let $y_1,y_2,y_3\in Y$. Then since $|\LL_x\psi|_\alpha\leq\zeta K\inf\LL_x\psi$, we have \begin{align*}
    \dfrac{Kd(y_1,y_2)^\alpha\LL_x\psi(y_3)-(\LL_x\psi(y_1)-\LL_x\psi(y_2))}{Kd(y_1,y_2)^\alpha\LL_x\phi(y_3)-(\LL_x\phi(y_1)-\LL_x\phi(y_2))}&\leq\dfrac{(K\sup\LL_x\psi+\zeta K\inf\LL_x\psi)d(y_1,y_2)^\alpha}{(K\inf\LL_x\phi-\zeta K\inf\LL_x\phi)d(y_1,y_2)^\alpha}.
\end{align*} Thus, $B(\LL_x\psi,\LL_x\phi)\leq \frac{K\sup\LL_x\psi+\zeta K\inf\LL_x\psi}{K\inf\LL_x\phi-\zeta K\inf\LL_x\phi}$. A similar calculation gives a lower bound on $A(\LL_x\psi,\LL_x\phi)$. So by Lemma \ref{lem:hilMet}, we have \begin{align*}
\Theta(\LL_x\psi,\LL_x\phi)
    &\leq \log\Bigg(\dfrac{K\sup\LL_x\phi+\zeta K\inf\LL_x\phi}{K\inf\LL_x\phi-\zeta K\inf\LL_x\phi}\cdot\dfrac{K\sup\LL_x\psi+\zeta K\inf\LL_x\psi}{K\inf\LL_x\psi-\zeta K\inf\LL_x\psi}\Bigg)\\
    &\leq \log\Bigg(\dfrac{K(1+\zeta K\diam(Y)^\alpha)(1+\zeta)\inf\LL_x\phi}{K(1-\zeta)\inf\LL_x\phi}\Bigg)\\&\hspace{3cm}+\log\Bigg(\dfrac{K(1+\zeta K\diam(Y)^\alpha)(1+\zeta)\inf\LL_x\psi}{K(1-\zeta)\inf\LL_x\psi}\Bigg)\\
    &\leq 2\log\Bigg(\dfrac{1+\zeta}{1-\zeta}\Bigg)+2\log(1+\zeta K\diam(Y)^\alpha)<\infty
\end{align*} This proves the existence of $M$.
\end{proof}

\begin{theorem}
\label{thm:potential}
Let $\Phi^{\sigma}_n(x)=\log\dfrac{\langle\LL_x^{n+1}\mathbbm{1},\sigma\rangle}{\langle\LL_{fx}^{n}\mathbbm{1},\sigma\rangle}.$ There exists $0<\tau<1$ and $C_1>0$ such that for all $k\in\NN,\ n,m\geq k,\ x\in X$, and any probability measures $\sigma_n$ on $Y_{f^nx}$ and $\sigma_m$ on $Y_{f^mx}$, we have \[\big|\Phi_n^{\sigma_n}(x)-\Phi_m^{\sigma_m}(x)\big|\leq C_1\ \tau^k.\] Thus, $\Phi(x)=\lim_{n\to\infty}\Phi_n^{\sigma_n}(x)$ exists  and $\big|\Phi_n^{\sigma_n}(x)-\Phi(x)\big|\leq C_1\tau^n$.
\end{theorem}

\begin{proof}
Fix $x\in X$. Suppose $n,m\geq k \geq 1$. Then \begin{align*}
\big|\Phi_n^{\sigma_n}(x)-\Phi_m^{\sigma_m}(x)\big|&=\Bigg|\log\dfrac{\langle\LL_x^{n+1}\mathbbm{1},\sigma_n\rangle}{\langle\LL_{fx}^{n}\mathbbm{1},\sigma_n\rangle}-\log\dfrac{\langle\LL_x^{m+1}\mathbbm{1},\sigma_m\rangle}{\langle\LL_{fx}^{m}\mathbbm{1},\sigma_m\rangle}\Bigg| \\
&=\Bigg|\log\dfrac{\langle\LL_{fx}^{k-1}(\LL_{x}\mathbbm{1}),\sigma_{fx,n}\rangle\langle\LL_{fx}^{k-1}\mathbbm{1},\sigma_{fx,m}\rangle}{\langle\LL_{fx}^{k-1}\mathbbm{1},\sigma_{fx,n}\rangle\langle\LL_{fx}^{k-1}(\LL_{x}\mathbbm{1}),\sigma_{fx,m}\rangle}\Bigg|
\end{align*} where $\sigma_{fx,n}=(\LL_{f^{k+1}x})^*\cdots(\LL_{f^nx})^*\sigma_n$. By Lemma~\ref{sec:hilMet}, we see that \[\big|\Phi_n^{\sigma_n}(x)-\Phi_m^{\sigma_m}(x)\big|\leq \Theta(\LL_{fx}^{k-1}(\LL_{x}\mathbbm{1}),\LL_{fx}^{k-1}\mathbbm{1}).\] Clearly, $\mathbbm{1}\in \Lambda^x_K$ for any $K>0$.  Then $\LL_{x}\mathbbm{1}\in \Lambda^{fx}_{\zeta K}$ by Lemma \ref{lem:ConeContraction}. Fix $K$ large and $M$ as in Lemma~\ref{lem:ConeContraction}. Set $\tau=\tanh{(M/4)}$. By Theorem~\ref{thm: Birkhoff}, we have \begin{align*}
    \Theta_{}(\LL_{fx}^{k-1}(\LL_{x}\mathbbm{1}),\LL_{fx}^{k-1}\mathbbm{1})&\leq \tau^{k-1}\Theta_{}((\LL_{x}\mathbbm{1}),\mathbbm{1})\leq \tau^{k-1}M.
\end{align*} Let $C_1=M/\tau$. Hence, the sequence $\{\Phi_n\}_{n\geq 0}$ is Cauchy and the limit exists at every $x\in X$.
\end{proof}

This proves the existence of $\Phi$.
\subsection{Fiber measures}
To completely understand the equilibrium state $\mu$ on $(\skprod,F)$, we need to understand how it gives weight to the fibers $\{Y_x\}_{x\in X}$. The first step is the following nonstationary Ruelle-Perron-Frobenius theorem adapted from \cite{CH}, whose proof we include here for completeness (see Hafouta \cite{H20} for a similar result when the base is invertible).

\begin{theorem}
\label{thm:fibMeas}
Let $F\colon\skprod\to\skprod$ satisfy \ref{ass:A1} and \ref{ass:A2}. For any H\"older $\varphi\colon \skprod\to\RR$ satisfying \eqref{ass:pot} and its associated family of fiberwise transfer operators $\{\LL_{x}\}_{x\in X}$, there exists a unique family of probability measures $\nu_x\in\MM(Y_x)$ such that for all $x\in X$, $\LL_{x}^*\nu_{fx}=\lambda_x\nu_x\quad \text{where\ } \lambda_x=\nu_{fx}(\LL_{x}\mathbbm{1})=e^{\Phi(x)}.$
\end{theorem}

Theorem \ref{thm:fibMeas} is a consequence of the following two propositions.

\begin{proposition}\label{thm:fibEigMeas}
Given any $x\in X$, $n\in\NN$, and $\sigma_n\in\mathcal{M}(Y_{f^nx})$, define $\nu_{x,n}\in\mathcal{M}(Y_{x})$ by $\nu_{x,n}=\dfrac{(\mathcal{L}_x^n)^*\sigma_n}{\langle\mathbbm{1},(\mathcal{L}_x^n)^*\sigma_n\rangle}$. Then exists $C_1>0$ such that with $\tau\in(0,1)$ as in Theorem \ref{thm:potential}, for all $k\in\NN$, $m,n\geq k$, $\psi\in\Lambda_K$, we have \[\Big|\int\psi d\nu_{x,n}-\int\psi d\nu_{x,m}\Big|\leq C_1\|\psi\|\tau^k.\] In particular, $\langle\psi,\nu_x\rangle:=\lim_{n\to\infty}\langle\psi,\nu_{x,n}\rangle$ exists and defines a probability measure $\nu_x$ on $Y_x$ with \[\Big|\int\psi d\nu_{x,n}-\int\psi d\nu_x\Big|\leq C_1\|\psi\|\tau^n.\]
\end{proposition}

\begin{proof}
Note that $\langle\psi,\nu_{x,n}\rangle=\dfrac{\langle\psi,(\LL^n_x)^*\sigma_n\rangle}{\langle\mathbbm{1},(\LL^n_x)^*\sigma_n\rangle}$. Let $b_k=\inf_y\dfrac{\LL^k_x\psi(y)}{\LL^k_x\mathbbm{1}(y)}$ and $c_k=\sup_y\dfrac{\LL^k_x\psi(y)}{\LL^k_x\mathbbm{1}(y)}$. 
Note that $\LL_x^k\psi\leq c_k\LL_x^k\mathbbm{1}$. So \[\langle\psi,\nu_{x,n}\rangle=\dfrac{\langle\psi,(\LL_x^n)^*\sigma_{n}\rangle}{\langle\mathbbm{1},(\LL_x^n)^*\sigma_{n}\rangle}=\dfrac{\langle\LL_x^k\psi,(\LL_x^{n-k})^*\sigma_{n}\rangle}{\langle\LL_x^k\mathbbm{1},(\LL_x^{n-k})^*\sigma_{n}\rangle}\leq\dfrac{c_k\langle\LL_x^k\mathbbm{1},(\LL_x^{n-k})^*\sigma_{n}\rangle}{\langle\LL_x^k\mathbbm{1},(\LL_x^{n-k})^*\sigma_{n}\rangle}=c_k.\] A similar computation shows that $b_k\leq\langle\psi,\nu_{x,n}\rangle$.  Then $b_k\leq\langle\psi,\nu_{x,n}\rangle\leq c_k$ for all $n\geq k$. Therefore, $|\langle\psi,\nu_{x,n}\rangle-\langle\psi,\nu_{x,m}\rangle|\leq c_k-b_k$ for all $n,m\geq k$. Lemma \ref{lem:ConeContraction} implies that $\Theta(\LL^k_x\psi,\LL^k_x\mathbbm{1})\leq\diam(\LL_x\Lambda_K)\tau^{k-1}\leq M\tau^{k-1}$. So $1\leq\dfrac{c_k}{b_k}\leq e^{M\tau^{k-1}}$. Thus, $b_k\leq c_k\leq b_ke^{M\tau^{k-1}}$ which implies that $c_k-b_k\leq b_k(e^{M\tau^{k-1}}-1)$. Moreover, for all $y\in Y$, we have \[\LL_x^k\psi(y)=\sum_{\overline{y}\in g_x^{-k}(y)}e^{S_k\varphi(x,\overline{y})}\psi(x,\overline{y})\leq\sum_{\overline{y}\in g_x^{-k}(y)}e^{S_k\varphi(x,\overline{y})}\|\psi\|=\|\psi\|\LL_x^k\mathbbm{1}(y).\] So $b_k\leq\|\psi\|$. Hence, \[|\langle\psi,\nu_{x,n}\rangle-\langle\psi,\nu_{x,m}\rangle|\leq c_k-b_k\leq\|\psi\|(e^{M\tau^{k-1}}-1).\] Thus, $\{\nu_{x,n}\}$ is a Cauchy sequence. Then there is a constant $C_1>0$ and \[|\langle\psi,\nu_{x,n}\rangle-\langle\psi,\nu_{x}\rangle|\leq C_1\|\psi\|\tau^n\] for all $n\geq 0$.
\end{proof}

\begin{proposition}
  Let $\{\nu_x\}$ be as in Proposition \ref{thm:fibEigMeas}. Then $\mathcal{L}_x^*\nu_{fx}=e^{\Phi(x)}\nu_x$.  
\end{proposition}

\begin{proof}
    For all $\psi\in C(Y_x)$, we have \begin{align*}
        \int\psi d(\LL_x^*\nu_{fx})&=\lim_{n\to\infty}\dfrac{\langle\LL_x\psi,(\LL^n_{fx})^*\sigma_{n+1}\rangle}{\langle\mathbbm{1},(\LL^n_{fx})^*\sigma_{n+1}\rangle}\\        &=\lim_{n\to\infty}\dfrac{\langle\LL_x^{n+1}\mathbbm{1},\sigma_{n+1}\rangle}{\langle\LL^n_{fx}\mathbbm{1},\sigma_{n+1}\rangle}\cdot\dfrac{\langle\LL_x^{n+1}\psi,\sigma_{n+1}\rangle}{\langle\LL^{n+1}_{x}\mathbbm{1},\sigma_{n+1}\rangle}=e^{\Phi(x)}.\qedhere
    \end{align*}
\end{proof}

Observe that \[\nu_{fx}(\LL_x\mathbbm{1})=\LL_x^*\nu_{fx}(\mathbbm{1})=e^{\Phi(x)}\nu_x(\mathbbm{1})=e^{\Phi(x)}.\] This completes the proof of Theorem \ref{thm:fibEigMeas}.

\subsection{Regularity of $\Phi$}
Now we will show that $\Phi$ is H\"older continuous. A direct consequence of Lemma \ref{lem:ConeContraction} is the following lemma which we will need to prove the H\"older continuity of $\Phi$. For convenience, we write \[\lambda_x^n=\lambda_x\lambda_{fx}\cdots\lambda_{f^{n-1}x}=e^{S_n\Phi(x)}.\]

\begin{lemma}\label{lem:transbound}
    Let $M$ be as in Lemma \ref{lem:ConeContraction}. Then $e^{-M}\lambda_x^n\leq\mathcal{L}^n_x\mathbbm{1}(y)\leq e^M\lambda_x^n$ for all $n\in\NN$ and $(x,y)\in X\times Y$. 
\end{lemma}

\begin{proof}
    Let $\overline{\varphi}_x=\varphi_x-\log \lambda_x$ and write $\overline{\mathcal{L}}^n_x\mathbbm{1}=\sum_{\overline{y}\in g_x^n(y)} e^{S_n\overline{\varphi}_x(y)}$. Theorem \ref{thm:fibEigMeas} gives $(\overline{\LL}_{x})^*\nu_{fx}=\nu_x$ for all $x\in X$. Inductively, we get that $(\overline{\LL}^{n}_x)^*\nu_{f^{n}x}=\nu_{x}$. Then for any $k,\ell$ \[\int \overline{\mathcal{L}}^{k}_x\mathbbm{1} d\nu_{f^kx}=\int \mathbbm{1} d(\overline{\mathcal{L}}^{k}_x)^*\nu_{f^kx}=\nu_x(Y_x)=1=\int\mathbbm{1}d\nu_{f^kx}.\] Let $\Lambda^+$ be the cone of strictly positive continuous functions on $\skprod$. Since $\Lambda_K\subset \Lambda^+$, the projective metrics of the two cones satisfy $\Theta^+(\phi,\psi)\leq\Theta(\phi,\psi)$. Write $\psi_k=\overline{\LL}_x^k\mathbbm{1}$. Then $\inf\psi_k\leq\mathbbm{1}\leq\sup\psi_k$ so $1\leq\frac{\sup\psi_k}{\inf\psi_k}\leq e^M$. We know that $\Theta^+(\psi_k,\mathbbm{1})\leq M$. This implies that $e^{-M}\leq\psi_k\leq e^M$ for all $k\in\NN$. Thus, \[e^{-M}\lambda_x^n\leq\mathcal{L}^n_x\mathbbm{1}(y)\leq e^M\lambda_x^n.\qedhere\]
\end{proof}
Let $n\in\mathbbm{N}$ and $x,x'\in X$, and $y\in Y$. Let $\mathcal{W}_n=\{1,\ldots,d\}^n$. By Lemma \ref{lem:FibPreimage}, we can write \[g^{-1}_{f^{n-1}x}(y)=\{y_1,\ldots,y_d\}\quad \text{and}\quad g^{-1}_{f^{n-1}x'}(y)=\{y'_1,\ldots,y'_d\}\] such that \[d((f^{n-1}x,y_k),(f^{n-1}x',y'_k))\leq L_kd_X(f^nx,f^nx')\] where $L_k=L$ if $1\leq k\leq q$ and $L_k=\gamma^{-1}$ if $q< k\leq d$. Continuing in this way, we get that \[g^{-n}_{x}(y)=\{y_w\in Y_x\colon w\in\mathcal{W}_n\}\quad \text{and}\quad g^{-n}_{x'}(y)=\{y'_w\in Y_{x'_w}\colon w\in \mathcal{W}_n\}\] such that for all $0\leq k\leq n$ \[d(F^{k}(x,y_w),F^{k}(x',y'_w))\leq L_{w_{k+1}}\cdots L_{w_{n}}d_X(f^nx,f^nx').\]

Let $m\leq\NN$ and $0<\iota<1$. A pair of inverse branches of length $n$ starting from $(f^nx,y)$ and $(f^nx',y)$ and labeled by $w\in \mathcal{W}_n$ is \emph{good} if for all $j\in\NN$ such that $jm\leq n$, we have \[\#\{n-jm< i\leq n\colon w_i\leq q\}\leq\iota jm.\] This means that the last $jm$ iterates of an orbit segment of length $n$ will be in the contraction region at most $\iota jm$ times. We will denote the collection of words corresponding to good trajectories by $$\mathcal{W}^\mathcal{G}_n=\mathcal{W}^\mathcal{G}_n(m)=\{w\in\mathcal{W}_n\colon \forall\  j\leq n/m, \#\{n-jm< i\leq n\colon w_i\leq q\}\leq\iota jm\}$$ and the collection of words for bad trajectories by $$\mathcal{W}^\mathcal{B}_n=\mathcal{W}^\mathcal{B}_n(m)=\{w\in\mathcal{W}_n\colon \exists\  j\leq n/m\ \text{such\ that}\ \#\{n-jm< i\leq n\colon w_i\leq q\}\geq\iota jm\}.$$

\begin{lemma}\label{lem:contract}
There is a $Q>0$ such that for all $m\in\mathbbm{N}$, if $(x,\overline{y})$ and $(x',\overline{y}')$ are preimages coded by a word in $\mathcal{W}_n^\mathcal{G}$, then \[d(F^k(x,\overline{y}),F^k(x',\overline{y}'))\leq Q^me^{-2c(n-k)}d(f^nx,f^nx')\] for all $0\leq k<n$.
\end{lemma}

\begin{proof}
Fix $m\in\NN$. Write $n-k=jm+i$ for $0\leq i<m$.  Since our preimage branches are assumed to be good, we get \begin{align*}
d(F^k(x,\overline{y}),F^k(x',\overline{y}'))&\leq L_{w_{k+1}}\cdots L_{w_{n}}d(F^{k+i}(x,\overline{y}),F^{k+i}(x',\overline{y}'))\\
&\leq L_{w_{k+1}}\cdots L_{w_{k+i}}(L^{\iota jm}\gamma^{-(1-\iota ){jm}})d(f^nx,f^nx')\end{align*}
Recall from \eqref{eqn:avg} that we can choose $c>0$ so that $0<\gamma^{-(1-\iota)}L^\iota<e^{-2c}<1$. Thus, \begin{align*}
d(F^k(x,\overline{y}),F^k(x',\overline{y}'))&\leq L^m e^{-2cjm}d(f^nx,f^nx')\\
&\leq (Le^{2c})^{m}e^{-2c(n-k)}d(f^nx,f^nx').\qedhere
\end{align*}
\end{proof}

Since it is assumed that $\varepsilon_\varphi<\log d-\log q$, then 
    $\frac{qe^{\varepsilon_\varphi}}{d}<1$. Fix $\varepsilon>0$ such that \begin{equation}\label{eqn:theta}\theta:=\frac{qe^\varepsilon e^{\varepsilon_\varphi}}{d}<1.\end{equation} The following lemma due to Varandas and Viana gives us a way to count the number of words that code bad trajectories of a given length. Let $I(\iota,n)=\{w\in\mathcal{W}_n\colon \#\{1\leq k\leq n\colon w_i\leq q\}\geq \iota n\}$.

\begin{lemma}\label{lem:bwords}
    Given $\varepsilon>0$, there exists a $\iota_0\in(0,1)$ such that \begin{equation*}
        \limsup_{n\to\infty}\frac{1}{n}\log\#I(\iota,n) <\log q+\varepsilon
    \end{equation*} for all $\iota\in(\iota_0,1)$. Therefore, there exists a $C>0$ such that $\#I(\iota,n) \leq Cq^ne^{\varepsilon n}$ for all $n$.
\end{lemma}

\begin{proof}
See Lemma 3.1 in Varandas and Viana \cite{VV08}.
\end{proof}

In what follows it will be convenient to write $a\colon \mathcal{W}_n\to Y$ so that $a(w)=y_w$ and $a'\colon \mathcal{W}_n\to Y$ so that $a'(w)=y'_w$. Lemma \ref{lem:FibPreimage} gives us bijections $b_j\colon \mathcal{W}_{jm}\to g^{-jm}_{f^{n-jm}x}(y)$ such that $b_j(v)=g_{x}^{n-jm}(a(uv))$ and $c_v\colon \mathcal{W}_{n-jm}\to g^{-(n-jm)}_{x}(b_j(v))$ such that $c_v(u)=a(uv)$ as well as their associated maps $b_j'$ and $c_v'$.

\begin{figure}[!h]
    \centering
    \includegraphics[width=.6\textwidth]{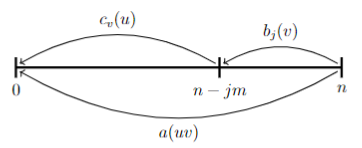}
    \caption{The maps $a$, $b_j$, and $c_v$ on words of corresponding lengths.}
    \label{fig:preimage}
\end{figure}

\begin{lemma}
\label{lem:BadGood}
Let $\theta$ as in \eqref{eqn:theta} above. There exists $C_2>0$ such that $$\sum_{w\in{\mathcal{W}_n^\mathcal{B}(m)}}e^{S_n\varphi_x(a(w))}\leq C_2\theta^m\sum_{w\in{\mathcal{W}_n^\mathcal{G}(m)}}e^{S_n\varphi_x(a(w))}$$ for all $m\in\NN, n\in \NN$, $x,x'\in X$, and $y\in Y$.
\end{lemma}

\begin{proof}
Let $x\in X$ and $y\in Y$.   For any $w\in\mathcal{W}_n^\mathcal{B}$, there is $1\leq j\leq n/m$ such that $w=uv$ for some $u\in\mathcal{W}_{n-jm}$ and $v\in I(\iota,jm)$. Thus,\begin{align*}
    \sum_{w\in{\mathcal{W}_n^\mathcal{B}}}e^{S_n\varphi_x(a(w))}&=\sum_{j=1}^{\lfloor n/m\rfloor}\sum_{v\in I(\iota,jm)}e^{S_{jm}\varphi_{f^{n-jm}x}(b_j(v))}\sum_{u\in\mathcal{W}_{n-jm}}e^{S_{n-jm}\varphi_x(c_v(u))}\\
        &\leq \sum_{j=1}^{\lfloor n/m\rfloor}\sum_{v\in I(\iota,jm)}e^{S_{jm}\varphi_{f^{n-jm}x}(b_j(v))}e^M\lambda_x^{n-jm}
\end{align*} by Lemma \ref{lem:transbound}. Note that for any $j\leq \frac{n}{m}$, \begin{align*}\sum_{w\in\mathcal{W}_n}e^{S_n\varphi_x(a(w))}&=\sum_{v\in \mathcal{W}_{jm}}e^{S_{jm}\varphi_{f^{n-jm}x}(b_j(v))}\sum_{u\in\mathcal{W}_{n-jm}}e^{S_{n-jm}\varphi_x(c_v(u))}\\&\geq e^{-M}\lambda_x^{n-jm}\sum_{v\in \mathcal{W}_{jm}}e^{S_{jm}\varphi_{f^{n-jm}x}(b_j(v))}.\end{align*} Lemma \ref{lem:bwords} implies that $\#I(\iota,n)\leq Cq^{n}e^{\varepsilon n}$ for all $n\geq 0$. Then since $\#\mathcal{W}_n$ is finite, \begin{align*}
\dfrac{\sum_{w\in{\mathcal{W}_n^\mathcal{B}}}e^{S_n\varphi_x(a(w))}}{\sum_{w\in\mathcal{W}_n}e^{S_n\varphi_x(a(w))}}&\leq \sum_{j=1}^{\lfloor n/m\rfloor}\dfrac{e^M\lambda_x^{n-jm}}{\sum_{w\in\mathcal{W}_n}e^{S_n\varphi_x(a(w))}}\sum_{v\in I(\iota,jm)}e^{S_{jm}\varphi_{f^{n-jm}x}(b_j(v))}\\
&\leq e^{2M}\sum_{j=1}^{\lfloor n/m\rfloor}\dfrac{\lambda_x^{n-jm}\sum_{v\in I(\iota,jm)}e^{S_{jm}\varphi_{f^{n-jm}x}(b_j(v))}}{\lambda_x^{n-jm}\sum_{v\in \mathcal{W}_{jm}}e^{S_{jm}\varphi_{f^{n-jm}x}(b_j(v))}}\\
&\leq e^{2M}\sum_{j=1}^{\lfloor n/m\rfloor}\dfrac{\#I(\iota,jm)}{\#\mathcal{W}_{jm}}e^{jm(\sup\varphi-\inf\varphi)}\\
&\leq e^{2M}\sum_{j=1}^{\lfloor n/m\rfloor}\dfrac{Cq^{jm}e^{jm\varepsilon}}{d^{jm}}e^{jm\varepsilon_\varphi}
\end{align*} where the last inequality holds by Lemma \ref{lem:bwords} and \eqref{ass:pot}. Let $\theta=\dfrac{qe^\varepsilon e^{\varepsilon_\varphi}}{d}<1$. Then
\[\dfrac{\sum_{w\in{\mathcal{W}_n^\mathcal{B}}}e^{S_n\varphi_x(a(w))}}{\sum_{w\in\mathcal{W}_n}e^{S_n\varphi_x(a(w))}}\leq e^{2M}\sum_{j=1}^{\infty}\bigg(\dfrac{qe^{\varepsilon}e^{\varepsilon_\varphi}}{d}\bigg)^{jm}= e^{2M}\bigg(\dfrac{\theta^m}{1-\theta^m}\bigg).\]
But $$\sum_{w\in\mathcal{W}_n}e^{S_n\varphi_x(a(w))}=\sum_{w\in{\mathcal{W}_n^\mathcal{B}}}e^{S_n\varphi_x(a(w))}+\sum_{w\in{\mathcal{W}_n^\mathcal{G}}}e^{S_n\varphi_x(a(w))}.$$ Choose $m$ such that $1-\theta^m\geq\frac{1}{2}$. Then \[\sum_{w\in{\mathcal{W}_n^\mathcal{B}}}e^{S_n\varphi_x(a(w))}\leq 2e^{2M}C\theta^m\bigg(\sum_{w\in{\mathcal{W}_n^\mathcal{B}}}e^{S_n\varphi_x(a(w))}+\sum_{w\in{\mathcal{W}_n^\mathcal{G}}}e^{S_n\varphi_x(a(w))}\bigg).\] So if we increase $m$ so that $2e^{2M}C\theta^m<\frac{1}{2}$, then \[\sum_{w\in\mathcal{W}_n^\mathcal{B}}e^{S_n\varphi_x(a(w))}\leq 4e^{2M}C\theta^m\sum_{w\in\mathcal{W}_n^\mathcal{G}}e^{S_n\varphi_x(a(w))}.\] This achieves the desired result.

%If $(x,\overline{y})$ and $(x',\overline{y}')$ start bad orbits of length $n$, then for some natural number $jm\leq n$ at least $\sigma jm$ of the iterates will be in $\mathcal{A}$. By Lemma 3.1 in Varandas Viana 2010, there are at most $Cq^{jm}e^{\varepsilon jm}d^{n-jm}$ such trajectories. Thus, in $g_x^{-n}y$, there are at most $\sum_{j\colon jm\leq n}Cq^{jm}e^{\varepsilon jm}d^{n-jm}$ bad trajectories.  Note that \begin{align*}
%    \dfrac{\Sigma_{(x,\overline{y})\in\mathcal{B}}e^{S_n\varphi(x,\overline{y})}}{\sum_{\overline{y}\in g_x^{-n}y}e^{S_n\varphi(x,\overline{y})}}
%     &\leq \dfrac{\#\mathcal{B} e^{\sup S_n\varphi(x,\overline{y})}}{\#\{g_x^{-n}y\}e^{\inf S_n\varphi(x,\overline{y})}}\\
%     &=\dfrac{\sum_{j\colon jm\leq n}Cq^{jm}e^{\varepsilon jm}d^{n-jm}}{d^n}e^{\sup S_n\varphi-\inf S_n\varphi}\\
%     &\leq e^{n(\sup\varphi-\inf\varphi)}\sum_{j=1}^\infty C\bigg(\dfrac{qe^\varepsilon}{d}\bigg)^{jm}\\
%    &\leq e^{n\varepsilon_\varphi} C'\bigg(\dfrac{qe^\varepsilon}{d}\bigg)^{m}
%\end{align*} for some $C'>0$. Let $C_2=C'e^{n\varepsilon_\varphi}$ and $\theta=\frac{qe^\varepsilon}{d}$
\end{proof}

Now we apply the above with various values of $m$ to prove H\"older continuity. First, a bound on $\Phi_n$.

\begin{lemma}\label{lem:nPhi}
    There exists $C_3>0$ and $\beta>0$ such that $$\big|\Phi_n(x)-\Phi_n(x')\big|\leq C_3d(f^nx,f^nx')^{\alpha\beta}$$ for all $x,x'\in X$ and $n\in\NN$.
\end{lemma}

\begin{proof}
First note that along good orbit pairs, we have by Lemma \ref{lem:contract} \begin{align}
    \big|S_n\varphi(x,\overline{y})-S_n\varphi(x',\overline{y}')\big|&\leq \sum_{k=0}^{n-1}|\varphi|_\alpha d(F^k(x,\overline{y}),F^k(x',\overline{y}'))^\alpha\nonumber\\ 
    &\leq \sum_{k=0}^{n-1}|\varphi|_\alpha Q^{\alpha m}e^{-2c\alpha (n-k)}d(f^nx,f^nx')^\alpha\nonumber\\ 
    &\leq Q^{\alpha m}d(f^nx,f^nx')^\alpha\cdot\sum_{k=0}^{\infty}|\varphi|_\alpha e^{-2c\alpha (n-k)}.\nonumber
\end{align}
Let $V=\sum_{k=0}^{\infty}|\varphi|_\alpha e^{-2c\alpha (n-k)}$. Then \begin{equation}
    \big|S_n\varphi(x,\overline{y})-S_n\varphi(x',\overline{y}')\big|\leq VQ^{\alpha m}d(f^nx,f^nx')^\alpha.
    \label{eqn:BirkSum}
\end{equation}
For convenience, we write $\Sigma_\mathcal{G}=\sum_{w\in\mathcal{W}_n^\mathcal{G}}e^{S_n\varphi_x(a(w))}$ and $\Sigma_\mathcal{B}=\sum_{w\in\mathcal{W}_n^\mathcal{B}}e^{S_n\varphi_x(a(w))}$ as well as $\Sigma_\mathcal{G}'$ and $\Sigma_\mathcal{B}'$ for the sums of the preimages associated to $a'(w)$. By Lemma \ref{lem:BadGood}, we get that  
\begin{align}\label{eqn:transRatio}
    \dfrac{\mathcal{L}^n_x\mathbbm{1}(f^nx,y)}{\mathcal{L}^n_x\mathbbm{1}(f^nx',y)}&=\dfrac{\Sigma_\mathcal{G}+\Sigma_\mathcal{B}}{\Sigma_\mathcal{G}'+\Sigma_\mathcal{B}'}
        \leq\dfrac{\Sigma_\mathcal{G}(1+C_2\theta^m)}{\Sigma_\mathcal{G}'(1-C_2\theta^m)}\leq \dfrac{\Sigma_\mathcal{G}}{\Sigma_\mathcal{G}'}\cdot e^{C\theta^m}
\end{align} Note that by \eqref{eqn:BirkSum} $$\dfrac{\Sigma_\mathcal{G}}{\Sigma_\mathcal{G}'}=\dfrac{\sum_{w\in\mathcal{W}_n^\mathcal{G}}e^{S_n\varphi_x(a(w))}}{\sum_{w'\in\mathcal{W}_n^\mathcal{G}}e^{S_n\varphi_{x'}(a'(w))}}\leq \dfrac{\sum_{w'\in\mathcal{W}_n^\mathcal{G}}e^{VQ^{\alpha m}d(f^nx,f^nx')^\alpha}e^{S_n\varphi_{x'}(a'(w))}}{\sum_{w'\in\mathcal{W}_n^\mathcal{G}}e^{S_n\varphi_{x'}(a'(w))}}\leq e^{VQ^{\alpha m}d(f^nx,f^nx')^\alpha}.$$

Let $\Phi_n$ be as in Theorem \ref{thm:potential} for the delta measure on $Y$, $\delta_y\ (y\in Y)$. Then \begin{align*}
    \big|\Phi_n(x)-\Phi_n(x')\big|&=\bigg|\log\bigg(\dfrac{\mathcal{L}^{n+1}_x\mathbbm{1}(f^nx,y)}{\mathcal{L}^{n+1}_{x'}\mathbbm{1}(f^nx',y)}\cdot\dfrac{\mathcal{L}^n_{fx}\mathbbm{1}(f^nx,y)}{\mathcal{L}^n_{fx'}\mathbbm{1}(f^nx',y)}\bigg)\bigg|\\ &\leq 2\log\bigg(\dfrac{\Sigma_\mathcal{G}}{\Sigma_\mathcal{G}'}\cdot e^{C\theta^m}\bigg)
    \\ &\leq 2VQ^{\alpha m}d(f^nx,f^nx')^\alpha+2C\theta^m
\end{align*} where the second inequality holds due to \eqref{eqn:transRatio}.

 Let $\rho_1=\frac{\theta}{Q^\alpha}$ and note that $\rho_1<1$. Then there is a $k\in\NN$ such that \begin{equation}\label{eqn:rho}
    \rho_1^{k+1}\leq d(f^nx,f^nx')^\alpha\leq\rho_1^k.
\end{equation} Now set $m=k$. Let $\beta=\frac{\log\theta}{\log\rho_1}$ and note that \[\theta^m=e^{m\log\theta}=e^{\beta m\log\rho_1}=\rho_1^{\beta m}\leq\rho_1^{-\beta}d(f^nx,f^nx')^{\alpha\beta}.\] Thus, $Q^{\alpha m}d(f^nx, f^nx')^\alpha\leq\theta^m\leq \rho_1^{-\beta}d(f^nx,f^nx')^{\alpha\beta}$. Hence, letting $C_3=2(V+C)\rho_1^{-\beta}$ yields \[\big|\Phi_n(x)-\Phi_n(x')\big|\leq C_3d(f^nx,f^nx')^{\alpha\beta}.\qedhere\]\end{proof}

\begin{theorem}
The potential $\Phi$ constructed in Theorem~\ref{thm:potential} is H\"older continuous.
\label{thm:PhiHolder}
\end{theorem}

\begin{proof}
Let $C=\max\{C_1,C_3\}$. For any $n\geq0$, \begin{align*}
    \big|\Phi(x)-\Phi(x')\big|&\leq \big|\Phi(x)-\Phi_n(x)\big|+\big|\Phi_n(x)-\Phi_n(x')\big|+\big|\Phi_n(x')-\Phi(x')\big|\\
    &\leq 2C\tau^n+Cd(f^nx,f^nx')^{\alpha\beta}\\
    &\leq 2C\tau^n+C\Gamma^{\alpha\beta n}d(x,x')^{\alpha\beta}.\end{align*} where the second inequality follows from Theorem~\ref{thm:potential} and Lemma~\ref{lem:nPhi} and $\Gamma$ is the inherited Lipschitz constant for $f$.\\
    
    Similar to the argument in the proof of Lemma~\ref{lem:nPhi}, we need to adjust the H\"older exponent to establish our bound. Let $\rho_2=\frac{\tau}{\Gamma^{\alpha\beta}}$. Then there is a $k$ such that $\rho_2^{k+1}\leq d(x,x')^{\alpha\beta}\leq \rho_2^k$. Let 
$n=k$ and $\eta=\frac{\log\tau}{\log\rho_2}$. Then \[\tau^n=e^{n\log\tau}=e^{\eta n\log\rho_2}=\rho_2^{\eta n}\leq\rho_2^{-\eta}d(x,x')^{\alpha\beta\eta}.\] So $\Gamma^{\alpha\beta}d(x,x')^{\alpha\beta}\leq \tau^n\leq \rho_2^{-\eta}d(x,x')^{\alpha\beta\eta}$. Therefore, \[\big|\Phi(x)-\Phi(x')\big|\leq 2C\rho_2^{-\eta}d(x,x')^{\alpha\beta\eta}.\qedhere\]
\end{proof}

\section{Conditional Measures of Equilibrium States}
\label{sec:fibers}

Let $\LL_\Phi\colon C(X)\to C(X)$ be defined by \[\LL_\Phi\xi(x)=\sum_{\Bar{x}\in f^{-1}x}e^{\Phi(\overline{x})}\ \xi(\Bar{x})\] for any $\xi\in C(X)$. Since $f$ is uniformly expanding on $X$, Theorem \ref{thm:PhiHolder} implies that there is a unique equilibrium state that can be obtained via $\LL_\Phi$. See \cite{W78} for details.

\begin{theorem}
\label{thm:TransEigen}
For any H\"older $\Phi\colon X\to\RR$, the following hold:
 \begin{enumerate}
     \item There exists a real number $\hat{\lambda}>0$ and $\hat{\nu}\in \MM(X)$ such that $\LL_\Phi^*\hat{\nu}=\hat{\lambda}\hat{\nu}$.
     
     \item There exists a unique $\hat{h}\in C(X)$ such that $\LL_\Phi\hat{h}=\hat{\lambda}\hat{h}$ and $\int_X\hat{h}(x)d\hat{\nu}(x)=1$.
     \item The unique equilibrium state for $\Phi$ is $\Bar{\mu}=\hat{h}\hat{\nu}$.
 \end{enumerate}
\end{theorem}

\vspace{5mm} 
We shall show that $\Bar{\mu}=\hat{\mu}=\mu\circ\pi_X^{-1}$ and construct the family of measures $\{\mu_x\}_{x\in X}$. To do this, we first prove the following lemmas.

\begin{lemma}
For any $\psi\in C(\skprod)$, the map $x\mapsto \LL^x_{\varphi}\psi_x$ is continuous with respect to the usual topology.
\label{lem:TransCont}
\end{lemma}

\begin{proof}
 Let $\psi\in C(\skprod)$, $x,x'\in X$, and $y\in Y$. Then \[\big|{\LL_{x}}\psi_x(y)-{\LL_{x'}}\psi_{x'}(y)\big|
   \leq \sum_{\overline{y}\in g_x^{-1}y}\Big(e^{\varphi(x,\overline{y})}\big|\psi(x,\overline{y})-\psi(x',\overline{y}')\big|+\|\psi\|_\infty\ \big|e^{\varphi(x,\overline{y})}-e^{\varphi(x',\overline{y}')}\big|\Big).\] Fix $\epsilon>0$. Let $M_1=\sup_{y\in Y}\{\LL_x\psi(y)\}$. By Lemma \ref{lem:FibPreimage}, we can choose $\delta>0$ small enough such that if $d_1(x,x')<\delta$, then for all $\overline{y}\in g_x^{-1}y$, \[\big|\psi(x,\overline{y})-\psi(x',\overline{y}')\big|<\dfrac{\epsilon}{2M_1}\quad \text{and}\quad \big|e^{\varphi(x,\overline{y})}-e^{\varphi(x',\overline{y}')}\big|<\dfrac{\epsilon}{2d\|\psi\|_\infty}.\]  
Therefore, we have that \begin{align*}
    \big|{\LL_{x}}\psi_x(y)-{\LL_{x'}}\psi_{x'}(y)\big|
      &\leq \dfrac{\epsilon}{2M_1}\sum_{\overline{y}\in g_x^{-1}y} e^{\varphi(x,\overline{y})}  + \|\psi\|_\infty\ \sum_{\overline{y}\in g_x^{-1}y} \dfrac{\epsilon}{2d\|\psi\|_\infty}\\&<\frac{\epsilon}{2}+\frac{\epsilon}{2}=\epsilon.\qedhere
    \end{align*}
\end{proof}

\begin{remark}This proof can be extended to hold for all iterates of the transfer operator.\end{remark}

\begin{lemma}\label{lem:measCont}
For every continuous $\psi\colon \skprod\to \RR$, the map $x\mapsto \nu_x(\psi_x)$ is continuous with respect to the usual topology.
\end{lemma}
 
\begin{proof}
Fix $y\in Y$ and let $\nu_{x,n}=\dfrac{(\LL^n_x)^*\delta_y}{\langle\mathbbm{1},(\LL^n_x)^*\delta_y\rangle}$ be as in Proposition \ref{thm:fibEigMeas}. So as shown there, $\nu_{x,n}\xrightarrow[\phantom{.}]{\text{wk}^*}\nu_x$.  For any $x,x'\in X$, \begin{align*}
   \Bigg|\int\psi d\nu_{x}-\int\psi d\nu_{x'}\Bigg|&\leq\Bigg|\int\psi d\nu_{x}-\int\psi d\nu_{x,n}\Bigg|+\Bigg|\int\psi d\nu_{x,n}-\int\psi d\nu_{x',n}\Bigg|\\ &\qquad\qquad+\Bigg|\int\psi d\nu_{x',n}-\int\psi d\nu_{x'}\Bigg|\\
   &\leq 2C_1\|\psi\|\tau^n + \Bigg|\int\psi d\nu_{x,n}-\int\psi d\nu_{x',n}\Bigg|
\end{align*} where the last inequality holds by Proposition \ref{thm:fibEigMeas}. Note that \[\int\psi d\nu_{x,n}=\dfrac{\langle\psi,(\LL^n_x)^*\delta_y\rangle}{\langle\mathbbm{1},(\LL^n_x)^*\delta_y\rangle}=\dfrac{\LL^n_x\psi(y)}{\LL^n_x\mathbbm{1}(y)}\] is continuous in $x$ by Lemma \ref{lem:TransCont}. Given $\epsilon>0$, choose $n$ sufficiently large so that $2C_1\|\psi\|\tau^n<\epsilon/2$ and $\delta>0$ such that $d(x,x')<\delta$ implies $|\int\psi d\nu_{x,n}-\int\psi d\nu_{x',n}|<\epsilon/2$. Then \[\Bigg|\int\psi d\nu_{x}-\int\psi d\nu_{x'}\Bigg|\leq 2C_1\|\psi\|\tau^n + \Bigg|\int\psi d\nu_{x,n}-\int\psi d\nu_{x',n}\Bigg|<\epsilon.\qedhere\] This proves continuity of $x\mapsto\nu_x(\psi_x)$.
\end{proof}

%Define a measure on $\skprod$ by \[\nu(\psi)=\int_X\int_{Y_x}\  \psi_x\ d\nu_x d\hat{\nu}(x).\]
%Thus, we can define a measure on $\skprod$ by $d\nu(x,y)=d\nu_x(y) d\hat{\nu}(x)$. Theorems~\ref{thm:TransEigen} and ~\ref{thm:fibMeas} allow us the following results about this measure $\nu$.

Define $I\colon C(X\times Y)\to C(X)$ by $(I\psi)(x)=\int_{Y_x}\psi(x,y)d\nu_x(x)$. Observe that for any $\eta\in\MM(X)$, we have \[\langle\psi,I^*\eta\rangle=\int_X(I\psi)(x)d\eta(x)=\int_X\int_Y\psi(x,y)d\nu_x(y)d\eta(x).\] So $\langle I\psi,\eta\rangle=\langle \psi,I^*\eta\rangle$ where $I^*\colon\mathcal{M}(X)\to\mathcal{M}(X\times Y)$ is defined by $I^*\eta=\int_X\nu_xd\eta(x)$.

\begin{theorem}\label{thm:commute}
    The operators $I$ and $I^*$ satisfy $I\circ\mathcal{L}_\varphi=\mathcal{L}_\Phi\circ I$ and $I^*\circ\mathcal{L}_\Phi^*=\mathcal{L}_\varphi\circ I^*$. That is, they make their respective diagrams below commute:\\

    \begin{minipage}{\textwidth}\begin{center}
            \begin{tikzcd}
C(X\times Y) \arrow[r, "\mathcal{L}_\varphi"] \arrow[d,"I"]
& C(X\times Y) \arrow[d, "I"] \\
C(X) \arrow[r, "\mathcal{L}_\Phi"]
& C(X)
\end{tikzcd} \qquad \begin{tikzcd}
\mathcal{M}(X\times Y)  
& \mathcal{M}(X\times Y) \arrow["\mathcal{L}_\varphi^*", l]\\
\mathcal{M}(X) \arrow[u,"I^*"]
& \mathcal{M}(X)\arrow[l, "\mathcal{L}_\Phi^*"]\arrow[u, "I^*"] 
\end{tikzcd}
    \end{center}

\end{minipage}
\end{theorem}

\begin{proof}
Given $\psi\in C(X\times Y)$, we have \begin{align*}
    I\circ\mathcal{L}_\varphi\psi(x,y)&=\int_{\skprod}\mathcal{L}_\varphi\psi(x,y)\ d\nu_x(y)\\
    &=\int_{\skprod}\sum_{\overline{x}\in f^{-1}x}\sum_{\overline{y}\in g_{\overline{x}}^{-1}y}e^{\varphi(\overline{x},\overline{y})}\psi(\overline{x},\overline{y})d\nu_x(y)\\
    &=\int_{\skprod}\sum_{\overline{x}\in f^{-1}x}(\mathcal{L}_{\overline{x}}\psi)(x,y)d\nu_x(y)\\
    &=\sum_{\overline{x}\in f^{-1}x}\langle\mathcal{L}_{\overline{x}}\psi,\nu_x\rangle\\
    &=\sum_{\overline{x}\in f^{-1}x}\langle\psi,(\mathcal{L}_{\overline{x}})^*\nu_x\rangle\\
    &=\sum_{\overline{x}\in f^{-1}x}e^{\Phi(\overline{x})}\langle\psi,\nu_{\overline{x}}\rangle=\sum_{\overline{x}\in f^{-1}x}e^{\Phi(\overline{x})}(I\psi)(\overline{x})\\
    &=(\mathcal{L}_\Phi\circ I)\psi(x)
\end{align*} Duality gives $I^*\circ\mathcal{L}_\Phi^*=\mathcal{L}_\varphi^*\circ I^*$.
\end{proof}

\begin{corollary}\label{cor:pressure}
$P(\Phi)=P(\varphi)$. Moreover, $\nu,\hat{\nu}$ and $h,\hat{h}$ satisfy $\nu=I^*\hat{\nu}$ and $\hat{h}=Ih$.
\end{corollary}

\begin{proof}
Note that $\nu,\ \hat{\nu},\ h,\ \hat{h}$ are uniquely determined as eigendata of their corresponding transfer operators. Moreover, by Theorem \ref{thm:commute}, \[\mathcal{L}_\Phi(Ih)=I\mathcal{L}_\varphi h=Ie^{P(\varphi)}h=e^{P(\varphi)}(Ih)\] and \[\mathcal{L}^*_\Phi(I^*\hat{\nu})=I^*\mathcal{L}_\Phi^* \hat{\nu}=I^*e^{P(\Phi)}\hat{\nu}=e^{P(\Phi)}(I\hat{\nu}).\] Thus, $Ih=\hat{h}$ and $I^*\hat{\nu}=\nu$. Since the eigenvalues are equal, we get $P(\varphi)=P(\Phi)$.

\end{proof}

Thus, given any $\psi\in C(\skprod)$, we have

\begin{align*}
    \int\psi\ d\mu&=\int\psi h\ d\nu\\
    &=\int_X\int_{Y}  (\psi\cdot h)({x},y)\ \ d\nu_x(y)d\hat{\nu}(x)\\
    &=\int_X\int_{Y}  \psi({x},y)\cdot \frac{h({x},y)}{\hat{h}(x)}\ \ d\nu_x(y)\hat{h}d\hat{\nu}(x)\\
    &=\int_X\int_{Y_{{x}}}  \psi({x},y)\ \ d\mu_x(y)d\Bar{\mu}(x)
\end{align*} where $\mu_x$ is defined by $\frac{d\mu_x}{d\nu_x}=\frac{h(x,y)}{\hat{h}(x)}$. Note that by Corollary \ref{cor:pressure} \[\mu_x(Y_x)=\int_{Y_x} \frac{h(x,y)}{\hat{h}(x)}d\nu_x(y)=\frac{(Ih)(x)}{\hat{h}(x)}=\frac{\hat{h}(x)}{\hat{h}(x)}=1.\] Therefore, $\Bar{\mu}=\hat{\mu}$ and $\{\mu_x\}_{x\in X}$ is the unique family of conditional measures for $\mu$.

\vspace{5mm}

%\begin{proposition}
%Let $\hat{h}$ be as in Theorem \ref{thm:TransEigen} and $h$. For any $x\in X$ and $y\in Y$, define $\Tilde{h}_x(y):=\frac{h(x,y)}{\hat{h}(x)}$. Then \[\sum_{\overline{x}\in f^{-1}x}\LL_{\overline{x}}\Tilde{h}_{\overline{x}}(y)\cdot \hat{h}(\overline{x})=\bigg(\sum_{\overline{x}\in f^{-1}x}e^{\Phi(\overline{x})}\hat{h}(\overline{x})\bigg)\Tilde{h}_{{x}}(y) .\]
%\end{proposition}
%\foot{This seems trivial. It is just the product transfer operator for $h$. Did you mean some relation like $\LL_x\Tilde{h}_x=\lambda_xh_{fx}$? And $(g_x)_*\mu_x=\mu_{fx}$ }
%\begin{proof}

%\begin{align*}
%    \sum_{(\Bar{x},\Bar{y})\in F^{-1}(x,y)} e^{\varphi(\Bar{x},\Bar{y})}\ h(x,y)
%    &=\sum_{\Bar{x}\in f^{-1}(x)}\hat{h}(\Bar{x})\ \LL_{\Bar{x}}h_{\Bar{x}}(y)\\
%    &=\sum_{\Bar{x}\in f^{-1}(x)}\lambda_{\Bar{x}}\hat{h}(\Bar{x})\ h_{x}(y)
%    =h_{x}(y)\ \LL_{\Phi}\hat{h}(x)=\lambda h(x,y)
%\end{align*} \qedhere
%\end{proof}

\bibliography{references}
\bibliographystyle{acm}

\end{document}